\documentclass[12pt]{amsart}
\usepackage{amsmath,amsthm,amsfonts,amssymb,latexsym,enumerate,graphicx,mathrsfs,mathabx,cancel}
\usepackage[letterpaper]{geometry}
\usepackage[dvipsnames]{xcolor}
\usepackage{hyperref,comment}
\hypersetup{colorlinks, citecolor=blue, linkcolor=red}

\usepackage{todonotes}

\usepackage{tikz}
\usetikzlibrary{arrows,decorations.pathmorphing,decorations.pathreplacing}
\usetikzlibrary{shapes.geometric,positioning}
\usepackage{tikz-cd}

\newtheorem{thm}{Theorem}[section]

\newtheorem{prop}[thm]{Proposition} 
\newtheorem{coro}[thm]{Corollary}

\newtheorem{lemma}[thm]{Lemma}
\newtheorem{claim}[thm]{Claim}
\newtheorem*{questions*}{Questions}

\theoremstyle{definition}
\newtheorem{defn}[thm]{Definition}

\theoremstyle{remark}
\newtheorem{remark}[thm]{Remark}

\makeatletter
\@tfor\next:=abcdefghijklmnopqrstuvwxyzABCDEFGHIJKLMNOPQRSTUVWXYZ\do{%
	\def\command@factory#1{%
		\expandafter\def\csname cal#1\endcsname{\mathcal{#1}}
		\expandafter\def\csname frak#1\endcsname{\mathfrak{#1}}
		\expandafter\def\csname scr#1\endcsname{\mathscr{#1}}
		\expandafter\def\csname bb#1\endcsname{\mathbb{#1}}
		\expandafter\def\csname rm#1\endcsname{\mathrm{#1}}
		\expandafter\def\csname bf#1\endcsname{\mathbf{#1}}
	}
	\expandafter\command@factory\next
}
\makeatother

\newcommand{\ad}{{\rm ad}}

\newcommand{\aut}[1]{{\mathrm{Aut}\left(#1\right)}}
\newcommand{\Aut}[1]{{\mathrm{Aut}\left(#1\right)}}
\newcommand{\Out}[1]{{\mathrm{Out}\left(#1\right)}}
\newcommand{\Comm}[1]{{\mathrm{Comm}\left(#1\right)}}

\DeclareMathOperator {\GL}{GL}
\DeclareMathOperator {\PSL}{PSL}
\DeclareMathOperator {\Homeo}{Homeo}

\begin{document}

\author{Fran\c{c}ois Dahmani}
\address{Univ. Grenoble Alpes, CNRS, IF, 38000 Grenoble, France}
\email{francois.dahmani@univ-grenoble-alpes.fr}
\urladdr{https://www-fourier.univ-grenoble-alpes.fr/~dahmani}

\author{Mahan Mj}

\address{School of Mathematics, Tata Institute of Fundamental Research, Mumbai-40005, India}
\email{mahan@math.tifr.res.in, mahan.mj@gmail.com}
\urladdr{http://www.math.tifr.res.in/~mahan}
\title{Class number for pseudo-Anosovs} 

\thanks{Both authors warmly thank the Institut Henri Poincar\'e (IHP)  for hosting the program "Groups Acting on Fractals" in Spring 2022, the 
	  {Institut des Hautes Etudes Scientifiques} for  the first author's visiting Carmin position, and the Centre de Mathematiques Laurent Schwartz.  This work was supported by LabEx CARMIN, ANR-10-LABX-59-01. FD  is supported by ANR-22-CE40-0004 GoFR.
	MM is supported by  the Department of Atomic Energy, Government of India, under project no.12-R\&D-TFR-5.01-0500,  by an endowment of the Infosys Foundation.
	and by   a DST JC Bose Fellowship. } 


\begin{abstract}
    Given two automorphisms of a group $G$, one is interested in knowing whether they are conjugate in the automorphism group of $G$, or in the abstract commensurator of $G$, and how these two properties may differ. When $G$ is the fundamental group of a closed orientable surface, we present a uniform finiteness theorem for the class of pseudo-Anosov automorphisms.  We present an explicit example of a commensurably conjugate pair of pseudo-Anosov automorphisms of a genus $3$ surface, that are not conjugate in the mapping class group, and we also show that infinitely many pairwise non-commuting pseud-Anosov automorphisms have class number equal to one. In the appendix, we briefly survey the Latimer-MacDuffee theorem that addresses the case of automorphisms of $\mathbb{Z}^n$, with a point of view that is suited to an analogy with surface group automorphisms. 
\end{abstract}

\maketitle

\section{Introduction: commensurated conjugacy}

    When $G$ is a group (or any structure), a natural problem is to classify the conjugacy classes  in its automorphism group $\Aut G$.  A familiar example is  $G=\bbZ^n$, for which  $\Aut G \simeq \GL_n(\bbZ)$.
    
    The \emph{abstract commensurator}  $\Comm G$ of $G$  is the group of equivalence classes of isomorphisms between two finite index subgroups of $G$, where two such automorphisms $\phi_1, \phi_2$ are declared to be equivalent if they agree  on further finite index subgroups.

    There is a natural homomorphism  from $\Aut G$  to $\Comm G$. It need not be injective, but in many cases of interest it is. In general, $\Comm G$ is much larger than the image of $\Aut G$. In our familiar example, $\Comm{\bbZ^n} \simeq \GL_n(\bbQ)$.  Of course, with linear algebra, conjugation in $\GL_n(\bbQ)$ is more easily understood than in $\GL_n(\bbZ)$. {(See \cite{studenmund} for instance for more
    details on commensurators in solvable Lie groups.)}

    For a group $G$, and $\phi \in \Aut G$, {let $[\phi]$ denote the image of $\phi$ in $\Comm G$. We define the \emph{commensurated-conjugacy class}           of $\phi$
    to be the set of automorphisms of $G$ that are conjugate to $[\phi]$} in $\Comm G$, and we say that these automorphisms are commensurably conjugate to $\phi$.  Any commensurated-conjugacy class is a union of  $\Aut G$-conjugacy classes.

    \begin{questions*} When are the commensurated-conjugacy classes strictly larger than $\Aut G$-conjugacy classes?   When do they consist of finitely many  $\Aut G$-conjugacy classes?
    \end{questions*}

    We call the number of $\Aut G$-conjugacy classes in the commensurated-conjugacy class of $\phi$,  the \emph{class number} of $\phi$. This terminology is  suggested and supported by our familiar example, as we explain now. 

   Let $G = \mathbb{Z}^n$, so that $\Aut G = \GL_n(\bbZ)$, and $\Comm G = \GL_n (\bbQ)$.  Choose $\phi \in \GL_n(\bbZ)$ such that its characteristic polynomial $\chi \in \bbZ[X]$ is irreducible over $\bbZ$. All  $\psi\in \GL_n(\bbZ)$ with the same characteristic polynomial as $\phi$ are  conjugate to $\phi$ in $\GL_n(\bbC)$, hence in $\GL_n(\bbQ)$ as well. Thus, they are commensurably conjugate.   The Latimer-MacDuffee theorem \cite{LMD, T} states that the $\GL_n(\bbZ)$-conjugacy classes of such elements are in correspondence with the ideal classes of the ring $\bbZ[X]/(\chi)$. Their cardinality, which is the order of the ideal class monoid of the ring,  is called the  \emph{class number} of the ring.  In other words, the class number of $\phi$ in our sense, is the class number of the ring $\bbZ[X]/(\chi)$.    
   It is  finite for all $\phi \in \GL_n(\bbZ)$, larger than the class number of the  ring of integers $\calO_{\bbQ[X]/(\chi)}$,  and it is greater than $2$ if and only if the ring is not a principal ideal domain.

    It is worth mentioning that  finiteness of class numbers is not true in general. A countable group in which the class number of an automorphism is not finite is the infinite direct product of copies of $\bbZ/3$: the automorphisms that are identity on finitely many copies, and flips on all the other copies, are all commensurably equal, but not conjugate. 
    
\bigskip

    We turn our attention to non-abelian counterparts of the groups  $\bbZ^n$. There are two celebrated and important such classes: finite rank non-abelian free groups, and fundamental groups of surfaces of higher genus. In this work we will consider the {latter}.   
    
    Consider a closed orientable surface $\Sigma$ of genus $g\geq 2$. The extended mapping class group of $\Sigma$ is the homeomorphism group of $\Sigma$ quotiented by its component of identity $\Homeo(\Sigma)/\Homeo_0(\Sigma)$.  
    By the Dehn-Nielsen-Baer theorem 
    {\cite[Chapter 8]{fm-book}}, it is isomorphic to the  outer automorphism group of $\pi_1(\Sigma)$.   However in order to make sense of commensuration, one must work at the level of automorphisms  rather than outer automorphisms -- in particular  one needs a base point $p$ in $\Sigma$. We have:

\[ \Aut{\pi_1(\Sigma,p)} \twoheadrightarrow \Out{\pi_1(\Sigma)} \xrightarrow{\sim} MCG^*(\Sigma) \]

 The   Nielsen-Thurston classification of mapping classes of $\Sigma$ distinguishes finite order  mapping classes,  reducible mapping classes, and { pseudo-Anosov mapping classes. 
 It is appropriate to regard	pseudo-Anosov mapping classes as  the correct placeholders for truly irreducible mapping classes.} 
 
 \begin{defn}\label{def-pa}{
 We will say that an  automorphism is a \emph{pseudo-Anosov}  if its image in the mapping class group under the above surjective homomorphism is a pseudo-Anosov mapping class.}
 \end{defn}
 We establish the following uniform finiteness  for class numbers of pseudo-Anosov automorphisms.

    \begin{thm}\label{theo;intro}
        If $\phi$ is a pseudo-Anosov automorphism of the fundamental group of a closed orientable surface of genus $g\geq 2$, its class number is finite, bounded above by $((168(g-1))!)^{2g}$. 
    \end{thm} 
    
    We think that it is  striking that our bound does not depend on $\phi$.  For comparison, in the case of $\Aut{ \bbZ^2} \simeq \GL_2(\mathbb{Z})$, by \cite{LMD, T}, one encounters  class numbers of certain rings of integers of  real quadratic fields, which are conjectured to have interesting (but elusive) behavior: Gauss' class number problem conjectures that infinitely many such class numbers are $1$, but   these numbers are  unbounded when the field extension varies among real quadratic extensions, as proved by Montgomery and Weinberger \cite{MW} and quantitatively suggested by Cohen and Lenstra's heuristics \cite{CL84}, and its revision by Bartel and Lenstra \cite{BL20}. 
    We explain in the appendix that the class numbers of elements of $\GL_2(\bbZ)$ are indeed finite but unbounded, an argument explained to us by Sara Checcoli.


    Finiteness of the class number of a pseudo-Anosov is not a surprise.  Here is a simple argument. Given $\phi$ a pseudo-Anosov automorphism on a surface $\Sigma$, the stretch factor, or entropy, { $\displaystyle \lim_{n\to \infty} \log |\phi^n(\gamma)|/n $ does not depend  on $\gamma\neq 1$} in the fundamental group, nor on the metric up to quasi-isometry. Hence it is an invariant of the commensurated-conjugacy class of $\phi$.  It also equals the translation length of $\phi$ in the Teichm\"uller space of the associated surface, and there are finitely many conjugacy classes of a pseudo-Anosov that can  have such translation length. { (See \cite{em} for much
    	finer results on  counting  closed geodesics in moduli space.) Hence there can only be finitely many conjugacy classes of $\phi$ in the commensurated-conjugacy class of $\phi$. } This argument does not, however,  provide any uniform bound. 
    
  Our upper bound is very likely non-optimal. It seems that an optimal bound would be given in terms of the value of a  subgroup-growth function of certain orbifolds. These  are difficult to estimate sharply, see \cite{LS}.    
    
    Examples of commensurably conjugate automorphisms that are not conjugate are not immediate. Our result  and methods however suggest a `recipe' for producing examples of non-conjugate pseudo-Anosov automorphisms that are commensurably conjugate. 
     {We look at the induced action on the abelianization: elements that are already conjugate in $\Aut G$ would necessarily be conjugate in  $\Aut{G/[G,G]}$, and in the latter group,  linear algebra tools may be applied to disprove the existence of such a conjugacy. }

    \begin{thm}\label{theo;example}
        Let $\Sigma$ be a closed orientable surface of genus $3$, with a base point $p$, and $G = \pi_1(\Sigma, p) $.   There exist $\phi$ and $\psi$ pseudo-Anosov automorphisms of $G$ that are commensurably conjugate but whose images in $\Aut {G/[G,G]} $ are not commensurably conjugate. In particular $\phi$ and $\psi$ are not conjugate in $\Aut G$.
    \end{thm}

    Gauss famously conjectures (in his `class number problem')  that infinitely many real quadratic extensions of $\bbQ$ have class number equal to one \cite{stark}. 
    We 
    observe that in our non-abelian analogue, the corresponding question can be settled. 
    
     \begin{thm}\label{theo;classnumberone}
        There are infinitely many commensurated classes of  automorphisms of hyperbolic $2$-orbifolds, that have class number one.
    \end{thm}
 
    \medskip
    
    \paragraph{\bf Summary of the paper.} After introducing the mapping tori involved and their relation to commensurated conjugacy (in Section 2), we discuss in Section 3 their representations as lattices in $\PSL_2(\bbC)$. We show that, for different commensurably conjugate automorphisms, the groups of the mapping tori have representations that  commensurate the image of the surface group in $\PSL_2(\bbC)$. We then use a result of Leininger, Long and Reid \cite{LLR}  on the discreteness of the commensurator in {$\PSL_2(\bbC)$} of such an image, to obtain that all such representations land in a single lattice, and we  obtain Proposition \ref{prop;simple_bound}. After finishing this work, {we learned of the related, but  more restrictive  notion of fibered commensurability investigated by} Calegari, Sun and Wang \cite[Section 2]{Cal}. We then prove our main theorem at the end of Section 3. In Section 4 we provide an explicit example, and a recipe for general examples. In Section 5 we address the class number one situation, proving Theorem \ref{theo;classnumberone}. We thought it worthwhile to attach an appendix, in which the algebraic number theory relevant to the situation of $\bbZ^n$ is briefly presented, in a way that allows for an analogy with geometry and the case of  surface group automorphisms. We briefly survey orders and ideal classes, provide an argument, explained to us by Sara Checcoli, as to why the class numbers of matrices in $\GL_2(\bbZ)$ are unbounded. We also explain a variant of Latimer-MacDuffee's proof  that allows for a geometric perspective on the situation, and is suited to an  analogy with surface group automorphisms. Finally, we explain this analogy and its limitations.
    
\medskip

    \paragraph{Acknowledgments.}	{The authors thank the referee for valuable comments, and suggestions. 
    The first author would like to thank Edgar Bering and  Sara Checcoli for discussions. } 


\section{Fundamental groups of mapping tori}

  The aim of this section is to prove Lemma \ref{lem;fibered_mapping_tori_commensurated} which converts the problem of
understanding commensurated-conjugacy classes of automorphisms of a group $G$ to commensurability of the associated semi-direct product by $\mathbb{Z}$. The one extra criterion that is vital in Lemma \ref{lem;fibered_mapping_tori_commensurated} is that the commensuration of the
associated semi-direct product also commensurates $G$.

    \begin{lemma} \label{lem;the_good_subgroup_and_m}

        If $ \phi, \psi$ are automorphisms of $G$ in the same commensurated-conjugacy class, there exists $H, H'$ of finite index in $G$, 
        and an isomorphism $\alpha : H\to H'$  such that,  $\psi(H)=H, \phi(H') = H'$,   $\psi|_{H} = \alpha^{-1} \circ   \phi \circ \alpha|_{H} $, where $|_H$ denotes  restriction to $H$ .  
    \end{lemma}

    \begin{proof}
        Let $[\alpha] \in \Comm G$ such that in
         $\Comm G$, 
         $[\psi] = {[\alpha]}^{-1} 
          [\phi]  [\alpha]$. 
         Realize $[\alpha]$ by an isomorphism $\alpha: T_1 \to T_2$ for $T_1, T_2$ finite index subgroups of $G$.

        Since $[\psi] = [\alpha]^{-1}  [\phi]  [\alpha]$, there is a further finite index subgroup $Y_1$ of $T_1$ on which  $\alpha^{-1} \circ \phi \circ \alpha = \psi$.  Let $Y_2= \alpha(Y_1)$. Observe that $\psi(Y_1)$ must be in $T_1$ but is perhaps not $Y_1$.

     Let $H$ denote the intersection of all subgroups in the $\Aut G$-orbit of $Y_1$, and $H'$ its image under $\alpha$. 
     {Then $H, H'$ are also of finite index.}  By construction, $\psi$ preserves $H$ (as well as every automorphism). We still have  $\alpha^{-1} \circ \phi \circ \alpha = \psi$ after to restriction  $H$. It follows that $\phi(H')= H'$. 
     
    \end{proof}

    Given $\phi \in \Aut G$, set  $\langle t_\phi \rangle$ to be an abstract infinite cyclic group, and  consider $\Gamma_\phi = G\rtimes_\phi \langle t_\phi \rangle$. In this semi-direct product, we say that $G$ is the \emph{ fiber}.
    
    \begin{remark}
     {This terminology is inspired by the motivating example in this 
    	paper, where $G=\pi_1(S)$ with $S$ a closed surface. Suppose 
    	that $\Phi: S \to S$ is a homeomorphism, so that $\phi$ equals the automorphism of $G$ induced by $\Phi$ after choosing a base-point appropriately. Let $M$ denote the 3-manifold given
    	by the \emph{mapping torus} of $\Phi$. 
    	Then  $\Gamma_\phi$ is the fundamental group of $M$. Thus, from
    	a topological point of view, the relevant object to be considered is the mapping torus $M$. But for the purposes of this
    	 section, it suffices to consider only the fundamental group of $M$. }
    \end{remark}

    \begin{lemma}\label{lem;fibered_mapping_tori_commensurated}
        If $ \phi, \psi$ are automorphisms of $G$ in the same commensurated-conjugacy class, then, 
        $\Gamma_{\phi}$ and $\Gamma_{\psi}$ are commensurable by a homomorphism commensurating the fiber.
    \end{lemma}

    \begin{proof}
        Let $H, H'$ be the subgroups obtained in the previous Lemma.  We make semi-direct products    $H\rtimes_{\psi} \langle s \rangle$ and  $H'\rtimes_{\phi} \langle q \rangle$. These two groups  are isomorphic by an isomorphism restricting to $\alpha$ on $H$ and sending $s$ {to} $q$.   Moreover these two groups embed  as finite index subgroups of $\Gamma_{\phi}$ and $\Gamma_{\psi}$ sending $s$ and $q$ respectively to $t_\psi$ and $t_\phi$.
    \end{proof}

\section{Surface groups and pseudo-Anosov automorphisms}

The aim of this section is to prove Theorem \ref{theo;intro}. We provide a quick roadmap through the section for the benefit of the reader.
Lemmas \ref{lem;lattice_representations} and \ref{lem;reps_rho} allow us to upgrade the study of an abstract commensurated-conjugacy class of pseudo-anosov automorphisms to the study of hyperbolic 3-manifolds arising as mapping tori. This makes use of Lemma \ref{lem;fibered_mapping_tori_commensurated}.  Lemma \ref{lem;rhoG_equal_psi_conjugate} then further upgrades the abstract commensurated-conjugacy class to the same conjugacy class in the automorphism group (and not just the outer automorphism group) of $G=\pi_1(S)$.

Propositions \ref{prop;fiber_thickens}  and \ref{prop;simple_bound} ensure that the images of $G$ under all
 the representations into $PSL(2,\mathbb{C})$ arising from Lemma \ref{lem;reps_rho} together generate
a \emph{finite index} supergroup $F_0$ of $G$, and hence the fundamental group of a hyperbolic 2-orbifold (Lemma \ref{lem;2orbi}). 

Corollary \ref{coro;1}, Lemma \ref{lem;rank} and Corollary \ref{coro;2} furnish effective estimates on the number of subgroups of $F_0$ of a fixed index.

These ingredients are finally assembled together to prove Theorem \ref{theo;intro} at the end of the Section.

    Recall that in a group $A$, an element $a$ is in the $A$-commensurator of a subgroup $B$ if $aBa^{-1} \cap B$ has finite index in $B$.  
  
    \begin{lemma}\label{lem;lattice_representations}
        Let $G$ be the fundamental group of a closed orientable surface of genus $g\geq 2$. If $\phi$ and $\psi$ are automorphisms of $G$ in the same commensurated-conjugacy class, and if $\phi$ determines a pseudo-Anosov mapping class, then $\psi$ does as well.  
    \end{lemma}
    
    \begin{proof}
        By Thurston's hyperbolization of mapping tori 
     {   \cite[Section 5]{thurston-hypstr2} (see case (iii) of \cite[Theorem 0.1]{thurston-hypstr2}),} the group $\Gamma_\phi$ is isomorphic to a lattice in $\PSL_2\bbC$, hence it is word-hyperbolic. Therefore $\Gamma_\psi$ is as well, by the quasi-isometry induced by the commensuration. Hence $\psi$ determines a pseudo-Anosov mapping class.
    \end{proof}

    \begin{lemma}\label{lem;reps_rho}
        There are faithful representations $\rho_\phi$, $\rho_\psi$ of  $\Gamma_\phi$ and $\Gamma_\psi$ into $\PSL_2(\bbC)$, such that for all $h\in H, \rho_\psi(h) = \rho_\phi (\alpha(h))  $ and  $\rho_\phi(t_\phi) = \rho_\psi(t_\psi)$. 
 
        Moreover, given $\rho_\phi$, the representation $\rho_\psi$ satisfying the above is unique.
    \end{lemma}

    \begin{proof}
        Consider the two uniform lattice representations $\rho_\phi$, $\rho_\psi$ of $\Gamma_\phi$ and $\Gamma_\psi$, respectively, into $\PSL_2\bbC$.

        The images contain isomorphic uniform sublattices, images of $H\rtimes \langle s\rangle$ and $H' \rtimes \langle q\rangle$ (following the notation of Lemma \ref{lem;fibered_mapping_tori_commensurated}). Through the isomorphism between these two groups, one may see them as lattice representations of the same group $H\rtimes \langle s\rangle$.  Therefore, by Mostow rigidity \cite[Theorem 12.1]{mostow}, one may conjugate $\rho_\psi$ so that for all $h\in H, \rho_\psi(h) = \rho_\phi (\alpha(h))  $  and also $\rho_\psi(q) = \rho_\phi(s)$.

        Uniqueness follows also from Mostow rigidity \cite[Theorem 12.1]{mostow}: two such representations $\rho_\psi, \rho'_\psi$ must be conjugate by a conjugator centralizing $\rho_\psi(H)$, but this {latter} group is Zariski dense in $\PSL_2(\bbC)$, hence its centralizer is the center of $\PSL_2(\bbZ)$ which is trivial.
    \end{proof}

    \begin{lemma}\label{lem;rhoG_equal_psi_conjugate}
        Assume  $\psi_1, \psi_2$ are both automorphisms of $G$, in the commensurated-conjugacy class of $\phi$.

        If $\rho_{\psi_1} (G) = \rho_{\psi_2} (G)$, then  $\psi_1, \psi_2$ are   in the same $Aut(G)$-conjugacy class. 
    \end{lemma}

    \begin{proof} 
     Since   $ \rho_{\psi_1} ( t_{\psi_1})= \rho_{\psi_2} ( t_{\psi_2})$, therefore 
        $\rho_{\psi_1}^{-1} \circ \rho_{\psi_2}$ is an automorphism of $G$ that tautologically conjugates the automorphism of $G$ given by $\rho_{\psi_1}^{-1} \circ  \ad_{\rho_{\psi_1} ( t_{\psi_1}) }    \circ \rho_{\psi_1}$ to  $\rho_{\psi_2}^{-1} \circ  \ad_{\rho_{\psi_1} ( t_{\psi_1}) }   \circ \rho_{\psi_2}$.
    \end{proof}

                The following lemma is not crucial, but we record it.
            
                \begin{lemma}\label{lem;equality_for_aut_conj}
                 
                   Assume  $\psi_1, \psi_2$ are both automorphisms of $G$, in the commensurated-conjugacy class of $\phi$.
                   If  $\psi_1, \psi_2$ are   in the same $Aut(G)$-conjugacy class of $\phi$, then  $\rho_{\psi_1} (G) =\rho_{\psi_2} (G)$.

                \end{lemma}
                
                In particular, if $\psi$ is conjugate to $\phi$ by $\alpha \in \Aut G$,   on the entire   $G$,  then $ \rho_\psi = \rho_\phi \circ \alpha  $ on $G$, while  in comparison with the previous case,  the equality held only after restriction to $H$.
                
                \begin{proof} The groups $\Gamma_{\psi_1}$ and $\Gamma_{\psi_2}$ are isomorphic by an isomorphism that sends the fiber to the fiber and $t_{\psi_1}$ to $t_{\psi_2}$. Therefore, there is a representation $\rho'_{\psi_2}$ of $\Gamma_{\psi_2}$ that has the same image as $\rho_{\psi_1}$. Further, the fibers have same image, as do $t_{\psi_1},t_{\psi_2}$. Since $\rho_{\psi_1} (t_{\psi_1}) = \rho_\phi(t_\phi)$, this representation satisfies the properties required for $\rho_{\psi_2}$ (as defined in Lemma \ref{lem;reps_rho}). By uniqueness of $\rho_{\psi_2}$ from Lemma \ref{lem;reps_rho},  $\rho_{\psi_2} = \rho'_{\psi_2}$ and we have the result.
                \end{proof}

  Let  $\psi$ be in the commensurated-conjugacy class of $\phi$. The representation $\rho_\psi$  sends $G$ to  a subgroup of $\PSL_2(\bbC)$ that is normalized by $\rho_\phi(t_\phi)$ and intersects $\rho_\phi (G)$ in a common finite index subgroup. After passing to a further common finite index subgroup $Y$, we may assume that $Y$ is normal in the group generated by $\rho_\phi (G) \cup \rho_\psi (G)$, hence the quotient   $ \langle \rho_\phi (G) \cup \rho_\psi (G)\rangle/Y  $  is isomorphic to a quotient of the abstract free product   $  \rho_\phi (G)/Y  *  \rho_\psi (G)/Y$. In principle this could still be infinite. However, this is not the case, as we argue below {in Proposition~\ref{prop;fiber_thickens} and the first claim of
  Proposition~\ref{prop;simple_bound}}.

    Observe that the limit set of  $\rho_\phi(G)$ in the sphere at infinity is the whole sphere $\partial \bbH^3$, since it is a normal subgroup in a lattice. Since it has infinite index in the said lattice, it has infinite co-volume in its action on $\mathbb{H}^3$. We may apply  \cite[Theorem 1.1]{LLR}, or equivalently \cite[Theorem 3.5]{Mj},  to obtain that the $(\PSL_2\bbC)$-commensurator of $\rho_\phi(G)$ is a uniform lattice $\Lambda_\phi$ containing  $\rho_\phi(\Gamma_\phi)$.

    \begin{prop}\label{prop;fiber_thickens}
   
        For all $\psi$ in the commensurated-conjugacy class of $\phi$, the group 
        $ \langle \rho_\phi (G) \cup \rho_\psi (G) \cup \{ \rho_\phi (t_\phi)\}\rangle$ is a uniform lattice that surjects onto $\bbZ$ with kernel  $ \langle \rho_\phi (G) \cup \rho_\psi (G) \rangle$.
  
        In particular,  $ \langle \rho_\phi (G) \cup \rho_\psi (G) \rangle $ is virtually a hyperbolic surface group. 
    \end{prop}
    
    \begin{proof}
        Every element of $\rho_\psi (G)$ is in the  $(\PSL_2\bbC)$-commensurator of $\rho_\phi(G)$ since the {latter} has a finite index subgroup which is a normal finite index subgroup of  $\rho_\psi (G)$. Therefore, $\rho_\psi (G) \subset \Lambda_\phi$, and $ \langle \rho_\phi (G) \cup \rho_\psi (G) \cup \{\rho_\phi (t_\phi)\}\rangle \subset \Lambda_\phi$. 
        Since it contains the uniform lattice $\rho_\phi (\Gamma_\phi)$, it is also a uniform lattice.  
   
        The obvious map sending $ \rho_\phi (t_\phi)$ to $1$ and  $  \rho_\phi (G) \cup \rho_\psi (G)$ to $0$ extends to a homomorphism $h$ to $\bbZ$. 
   
        By Selberg's lemma,     there is a torsion free finite index subgroup $F$ of   $\langle \rho_\phi (G) \cup \rho_\psi (G)\rangle$.  Observe, though it is not directly useful, that the finite index of $F$ is uniformly bounded, if Selberg's lemma is applied to $\Lambda_\phi$.    Considering the quotient of $\bbH^3$ by $F$, 
        the Tameness Theorem \cite{cg,agol}  and Canary's Covering Theorem \cite{canary-cover} imply that $F$ is a closed surface group.  
        Observe that the conclusion can be also drawn from Stallings' Theorem \cite{Stallings_fibration} applied to a torsion-free finite index subgroup of  $ \langle \rho_\phi (G) \cup \rho_\psi (G) \cup \{\rho_\phi (t_\phi)\}\rangle$, which projects onto $\bbZ$ with finitely generated kernel. 
            
        It follows that  $ \langle \rho_\phi (G) \cup \rho_\psi (G)\rangle $ is virtually   the fundamental group of a hyperbolic surface. {Further, $ \langle \rho_\phi (G) \cup \rho_\psi (G)\rangle $ is contained in a finite index supergroup $N$, the kernel of the homomorphism
        	$h$ constructed above. Hence the automorphism $ \ad_{\rho_\phi(t_\phi)} $ of $N$ extends both
       $\phi$ and $\psi$.} 
   
   \end{proof}

   \begin{prop}\label{prop;simple_bound}
        Given $\rho_\phi$, there is a  subgroup  $F_0$ of $\Lambda_\phi$ that contains $\rho_\phi(G)$ as a finite index subgroup that is   normalized by $\rho_\phi(t_\phi)$ and that contains all images $\rho_\psi(G)$ for $\psi$ in the commensurated conjugacy class of $\phi$.  
    
        The subgroup $F_0$ is  virtually a hyperbolic surface  group, and it contains finitely many  finite index subgroups isomorphic to $G$. 
    
        This number is an upper bound for the class number of $\phi$.
   
   \end{prop}
   
   \begin{proof}
        Take the collection of all $\rho_\psi(G)$ for all possible representation $\psi$ as above commensurating $\rho_\phi(G)$. Index them over the integers, let $G_0= \rho_\phi(G)$ and for each $i$ let $G_i=\langle G_{i-1} \cup \rho_{\psi_i}(G)\rangle$. This sequence of subgroups of $\Lambda_\phi$ is a sequence of groups containing   $\rho_\phi(G)$ as a finite index subgroup, and normalized by $ \rho_\phi(t_\phi)$. The previous argument shows that, up to bounded finite index,  this corresponds to a sequence of decreasing finite covers of hyperbolic surfaces, and thus must terminate, by decrease of the genus. Since the index is uniformly bounded, the sequence is stationary. Let $F_0$ be the union of all $G_i$, thus the first part of the lemma holds.  Considering Euler characteristic, all finite index subgroups of $F_0$ that are isomorphic to $G$ have the same index, and the second part is just counting subgroups with a  given finite index.   The third part is a consequence of Lemma \ref{lem;rhoG_equal_psi_conjugate}. 
   \end{proof}

   \begin{lemma}\label{lem;2orbi}
        The group $F_0$ is a hyperbolic $2$-orbifold group. 
        
        More generally, anytime a group $F_0$ is a finitely generated kernel of a homomorphism to $\bbZ$ of a  uniform lattice in $\PSL_2(\bbC)$, it is a hyperbolic $2$-orbifold group.
   \end{lemma}
   
   \begin{proof} As already mentioned, such a group $F_0$ contains as finite index subgroup, the fundamental group of a closed orientable surface of genus $\geq 2$.
   
        Therefore there exists   a finite index normal subgroup $K$ of $F_0$ that is a  hyperbolic surface group. One has an exact sequence \[1\to K \to F_0 \to Q \to 1 \] 
        where $Q$ is finite. Choosing a set-theoretic section $\xi: Q \to F_0$, we obtain a map $Q\to \Aut K$ by realizing each $\xi(q)$ by $\ad_{\xi(q)}$. This is not necessarily a homomorphism (since $\xi(q)$ could be changed by any element of $\xi(q)K$), but it  descends to a homomorphism $\varphi: Q\to \Out K$. 

        We first argue that this homomorphism is injective. If $q\in Q\setminus\{1\}$ is sent to the identity, then $\ad_{\xi(q)}$ equals  conjugation by some $k\in K$, hence $qk^{-1}$ is in the centralizer of $K$. However, as $K$ is Zariski dense in $\PSL_2(\bbC)$ (which has trivial center), it has trivial centralizer. 
   
        By the Dehn-Nielsen-Baer theorem,  $\Out K$ is isomorphic to the extended mapping class group of the surface $\Sigma$ of which $K$ is the fundamental group. By the Nielsen realization theorem \cite{Kerck},  there is a hyperbolic metric on $\Sigma$, a subgroup $S$ of the finite group of (metric) symmetries of $S$,   an identification $ K \stackrel{\simeq}{\to} \pi_1(\Sigma, v) $ (for a base point $v \in \Sigma$ with trivial $S$-stabilizer), and an isomorphism     $\varphi(Q) \stackrel{\sim}{\to} S$  in a way that the action of  $S $  by outer-automorphisms on $\pi_1(\Sigma, v)$ is the equivariant image of the action of $\varphi (Q)$ by outer-automorphisms on $K$. 
   
       Consider now the orbifold quotient $\overline \Sigma$ of $\Sigma$ under the action of $S$. We claim that its orbifold fundamental group (with base point $ \bar v$)
        is isomorphic to the extension $F_0$.
        Indeed, it is also a group extension of the form 
            \[1\to     \pi_1(\Sigma) \to \pi_1(\overline \Sigma) \to S \to 1\] 
        with the same representation $S\to \Out{ \pi_1(\Sigma)} $ as $\varphi$ for $F_0$. Since the center of $K$ is trivial,  the classification of extensions (see for instance \cite[Coro. IV(6.8)]{Brown})  
        indicates that there is only one extension of $K$ by $Q$, realizing $\varphi$, up to equivalence. Therefore $\pi_1(\overline \Sigma, \bar v) \simeq F_0$, and $F_0$ is a $2$-orbifold group.
   \end{proof}
   
    Being commensurable with a hyperbolic surface group, the orbifold of which $F_0$ is the fundamental group  is an orbifold that carries a hyperbolic metric. It must be of area greater than that of the smallest orbifold (the quotient of $\bbH^2$ by the the Coxeter triangle group $(2,3,7)$),
    which is $\pi/42$. On the other hand, if $\chi(G)$ is the Euler characteristic of $G$, it is the fundamental group of a hyperbolic surface of area $-2\pi\chi(G)$. We record this in the following.

   \begin{coro} \label{coro;1}
        The index of subgroups of $F_0$ that are of finite index and isomorphic to $G$ is bounded above by  $-84\chi(G)$, or $168 (g-1)$. Here $g$ is the genus of a surface $\Sigma$ with   fundamental group $G$. 
   \end{coro}
   
   \begin{lemma}\label{lem;rank}
       If $g$ is the genus of a surface $\Sigma$ with   fundamental group $G$,  then the rank of $F_0$ is at most $2g$. 
   \end{lemma}

    \begin{proof}
    The group $F_0$ is fundamental group of an orbifold $\calO$ of which the surface $\Sigma$ is a ramified cover. Let $\calO_{top}$ be the  underlying topological surface of $\calO$, and $\calO_{sing}$ the collection of singularities of $\calO$ on $\calO_{top}$. For $p$ a singularity let $e_p$ denote the order of its isotropy group.  Classical presentations of $2$-orbifolds show that the   rank of $F_0$ is at most the rank of the fundamental group of $\calO_{top}$  plus the cardinality of $\calO_{sing}$.  The surface $\Sigma$ is a branched cover of $\calO_{top}$. Let $N$ be the degree of the cover.  The Riemann-Hurwitz formula for branched covers of (topological) surfaces gives   
    $\chi(\Sigma) = N\chi(\calO_{top})-\sum_{\calO_{sing}} (e_p-1)$. Writing  the rank of the (classical) fundamental group of $\calO_{top}$ as $r_t$, one has $\chi(\Sigma)=N(2-r_t) - \sum(e_p-1)$. The rank $r$ of the orbifold group $F_0$ hence satisfies  
    $Nr\leq 2N - \chi(\Sigma) -N\sum(e_p-1) +N|\calO_{sing}| \leq 2N - \chi(\Sigma)$. Finally $r\leq 2+ 2 (g-1)/N =  \frac{2g}{N}+2-\frac{2}{N}$. This in turn  is less than $2g$ for $g\geq 2, N\geq 1$.   
    \end{proof}
    
    \begin{coro}\label{coro;2}
        The number of subgroups of index $k$ in $F_0$ is at most $(k!)^{2g}$.
    \end{coro}
    \begin{proof} For any index $k$ subgroup, $F_0$ acts by permutation on its $k$ cosets. Therefore,  the map $Hom(F_0,  \frakS_k) \to \{H, H<F_0\}$ that assigns to every $\pi: F_0\to \frakS_k $  the stabilizer of $\{1\}$ of the action,    surjects $Hom(F_0,  \frakS_k)$  onto the set of index $\leq k$ subgroups of $F_0$.    There are at most $(k!)^{rank(F_0)}$ such homomorphisms. By the previous Lemma \ref{lem;rank}, we obtain the desired bound. 
    \end{proof}

    A  slightly better formula for surfaces was actually given by A.D. Mednykh, \cite[Thm. 14.4.1]{LS}.   
   
   We thus obtain a proof of Theorem \ref{theo;intro}.
   
   \begin{proof}[Proof of Theorem \ref{theo;intro}]
 	Consider $\Sigma_g$ a genus $g\geq 2$ surface, and  $\phi$ a pseudo-Anosov automorphism of $G=\pi_1(\Sigma_g)$. Let $\rho_\phi$ a discrete faithful representation of $\Gamma_\phi= \pi_1(\Sigma_g) \rtimes_\phi \mathbb{Z}$ in $PSL_2(\mathbb{C})$.  Let $F_0$ given by  Proposition \ref{prop;simple_bound}. By  Corollary  \ref{coro;1}, the index of $F_0$ in all the images of $G$ by the representations $\rho_\psi$ defined before  Proposition \ref{prop;simple_bound}, for $\psi \in \Aut{G}$ in the commensurated conjugacy class of $\phi$ is bounded by $168(g-1)$. By   Corollary  \ref{coro;1}, the number of images of $G$ by the maps  $\rho_\psi$ is bounded by  $((168(g-1))!)^{2g}$. Finally, by Lemma \ref{lem;rhoG_equal_psi_conjugate}, 
 	this number is an upper bound for the number of $\Aut{G}$-conjugacy classes in the set 
 	of all $\psi \in \Aut{G}$ that are commensurably conjugate to $\phi$.
   \end{proof}

   However huge,  this bound is nonetheless independent of $\phi$, which was perhaps unexpected.  Hence, the bound obtained in Proposition \ref{prop;simple_bound} is  actually uniform over the elements of $\Aut G$ that define pseudo-Anosov mapping classes of the surface of which $G$ is fundamental group.

\section{Recipe for commensurably-conjugate mapping classes}
    
    In this Section we propose a recipe to construct examples, as suggested by the structure obtained in Proposition \ref{prop;simple_bound}. We will also produce an explicit example in Subsection \ref{subsec;ex}, thus proving Theorem \ref{theo;example}.

    Start with a hyperbolic surface, or $2$-orbifold $\Sigma$ with a base point $v$. Consider a single pseudo-Anosov mapping class $[\Phi]$, and realize it as an automorphism $\Phi$ of $\pi_1(\Sigma, v)$. 

    Consider two finite sheeted characteristic covers that are of the same degree but  do not correspond to subgroups in the same orbit under the automorphism group of $\pi_1(\Sigma, v)$. For instance, one can be an abelian  cover while the other is not.
    In fact, it is not necessary that the covers are characteristic, they only need to have fundamental groups preserved by $\Phi$. It could thus be any  cover of degree $k$, up to replacing $\Phi$ by $\Phi^{k!}$ (which   preserves all subgroups of index $k$). 
    
    Let $G_1, G_2$ be the subgroups of $\pi_1(\Sigma, v)$ that are fundamental groups of the respective covers $(\Sigma_1, v_1)$ and $(\Sigma_2, v_2)$. 
    
    The automorphism $\Phi$ {induces} $\phi$ and $\phi'$, automorphisms of subgroups $G_1, G_2$ {respectively}.  
Moving up to a common cover of $\Sigma_1, \Sigma_2$, corresponding to the intersection of the subgroups $G_1, G_2$, (which is preserved by $\Phi$), we obtain a  subgroup on which $\phi, \phi'$ coincide. 
Turn $\phi'$ into an automorphism $\psi$ of $G_1$ by picking an isomorphism $\alpha: G_1\to G_2$.

    \begin{claim}
        The automorphisms $\phi$ and $\psi$ of the group $G_1$  are in the same commensurated-conjugacy class in $\Comm {G_1}$.
    \end{claim}

    Write $H = G_1 \cap G_2$.   
    On $H$, the restrictions of $\phi$ and $\phi'$ coincide.   Since $\alpha \circ  \psi \circ \alpha^{-1} = \phi'$ on $G_2$, we have that $\phi|_H = \alpha \circ  \psi \circ \alpha^{-1}|_H$.  
   We thus have $\phi, \psi$ commensurably conjugated by $\alpha^{-1}$ on $H$.   
 It is not clear   what conditions would ensure that they are not conjugate in $\Aut{G_1}$, even though one might suspect that this happens somewhat generically:  any such conjugation would be a hidden symmetry for the mapping class of $\phi$ on its cover $\Sigma_1$, because it can then descend as $[\Phi]$  to $\Sigma$ through two different covering maps.

        \subsection{Explicit example}\label{subsec;ex}

        While it appears likely that a construction as above would give examples of non-conjugated automorphisms, an actual explicit example is necessary to settle the issue.
        
        \begin{defn}{
        Let $A \in \GL_n(\mathbb{R})$. If $(A-I)$ as an endomorphism
        of $\mathbb{R}^n$ has rank $r < n$, we say that $A$ is \emph{the identity plus a rank $r$ endomorphism}.}
        \end{defn}

         Our example below will be 
        a construction of  two mapping classes on a surface of genus three. 
        
        \begin{prop}
            Let $S$ be a closed orientable surface of genus three, and $p\in S$. There exist two pseudo-Anosov mapping classes of $S$ that are realized by automorphisms of the fundamental group $\pi_1(S, p)$ that are commensurably conjugate, but  are not conjugate in   $MCG(S)$. 
        \end{prop}
  
        In particular they are not conjugate in $Aut(\pi_1(S,p)$.
        
        \begin{proof}
         Consider $\Sigma$ a genus $2$ surface, with base point $v$, and two  curves that fill $\Sigma$, $\mu, \lambda$, with $\mu$ separating. Writing $\pi_1(\Sigma, v) =\langle \alpha, \eta, \gamma, \delta \, |\,  \alpha \beta \alpha^{-1} \beta^{-1}  = \gamma^{-1}\delta \gamma \delta^{-1} \rangle  $,  we specify  the curve $\mu$ to be the commutator and for the sake of being explicit, $\lambda = \gamma\beta\delta\gamma\beta\alpha^{-1}\delta\gamma\beta$, see Figure \ref{fig;lambda}.
         
        \begin{figure}[ht]
            \centering
                     \includegraphics[width=0.6\textwidth]{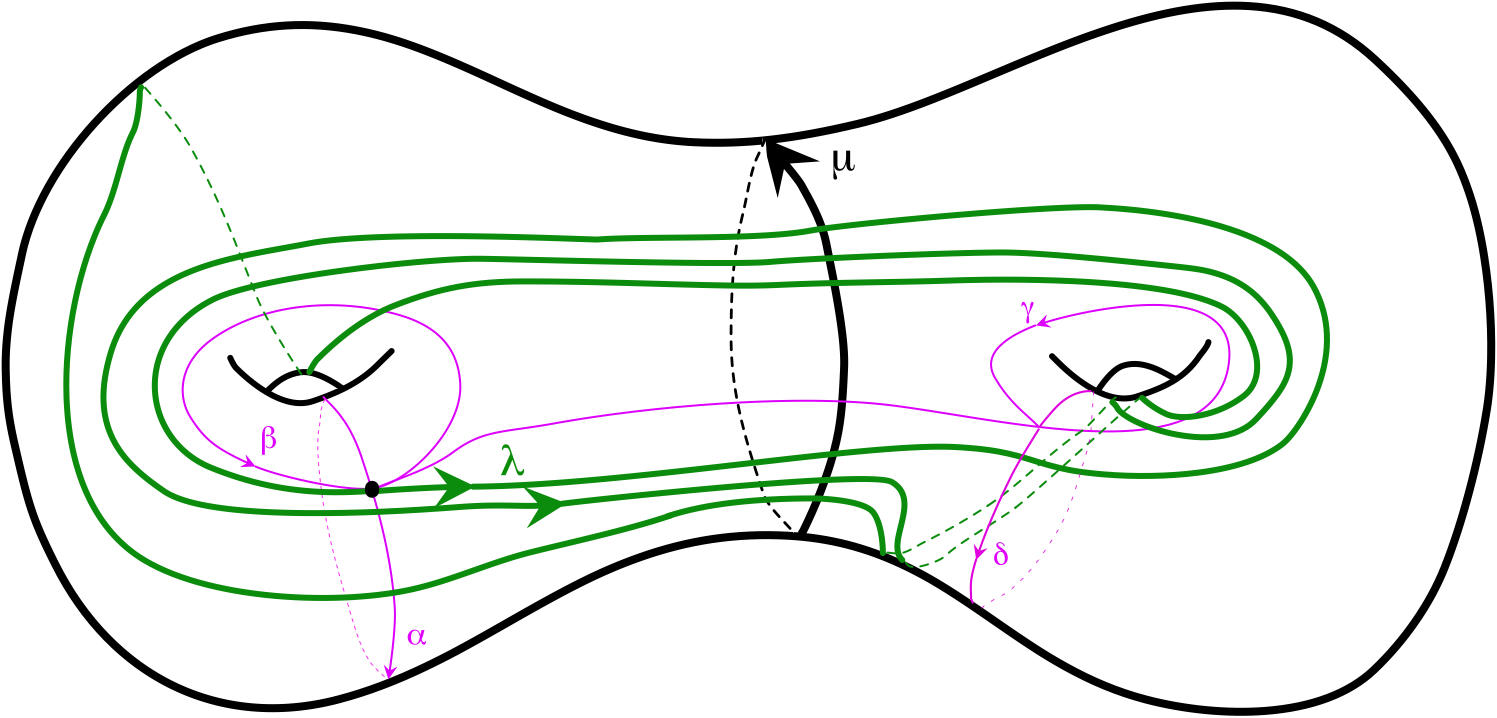}
         
            \caption{The curve $\lambda=\gamma\beta\delta\gamma\beta\alpha^{-1}\delta\gamma\beta$, in green, used for twisting. The curves $\alpha$ and $\beta$ are pink on the left, $\gamma, \delta$ are pink on the right, while $\mu$ s in black in the middle $\beta$ and $\gamma$ are the "horizontal" curves among the pink ones.}
            \label{fig;lambda}
        \end{figure}

         Consider the mapping classes obtained by the product of Dehn twists $\phi= \tau_{\gamma_1}^n\tau_{\lambda}^{2m}$. The construction of Penner-Thurston \cite{Penner} ensures that, for $n, m>\!\!>1$, this defines a pseudo-Anosov mapping class on $\Sigma$. 
         
         Consider the two 2-sheeted covers $\Sigma_1, \Sigma_2$ of $\Sigma$ whose fundamental groups are the kernels of the two homomorphisms $\pi_1(\Sigma, v) \to \bbZ/2$ given by \[\alpha, \delta \mapsto 0, \quad \beta, \gamma \mapsto 1\] on the one hand, and \[\alpha, \gamma, \delta \mapsto 0, \quad \beta \mapsto 1\] on the other. These are two surfaces of genus 3,  pictured in Figure \ref{fig;covers}.
         
         \begin{figure}[ht]
             \centering
             \includegraphics[width=0.8\textwidth]{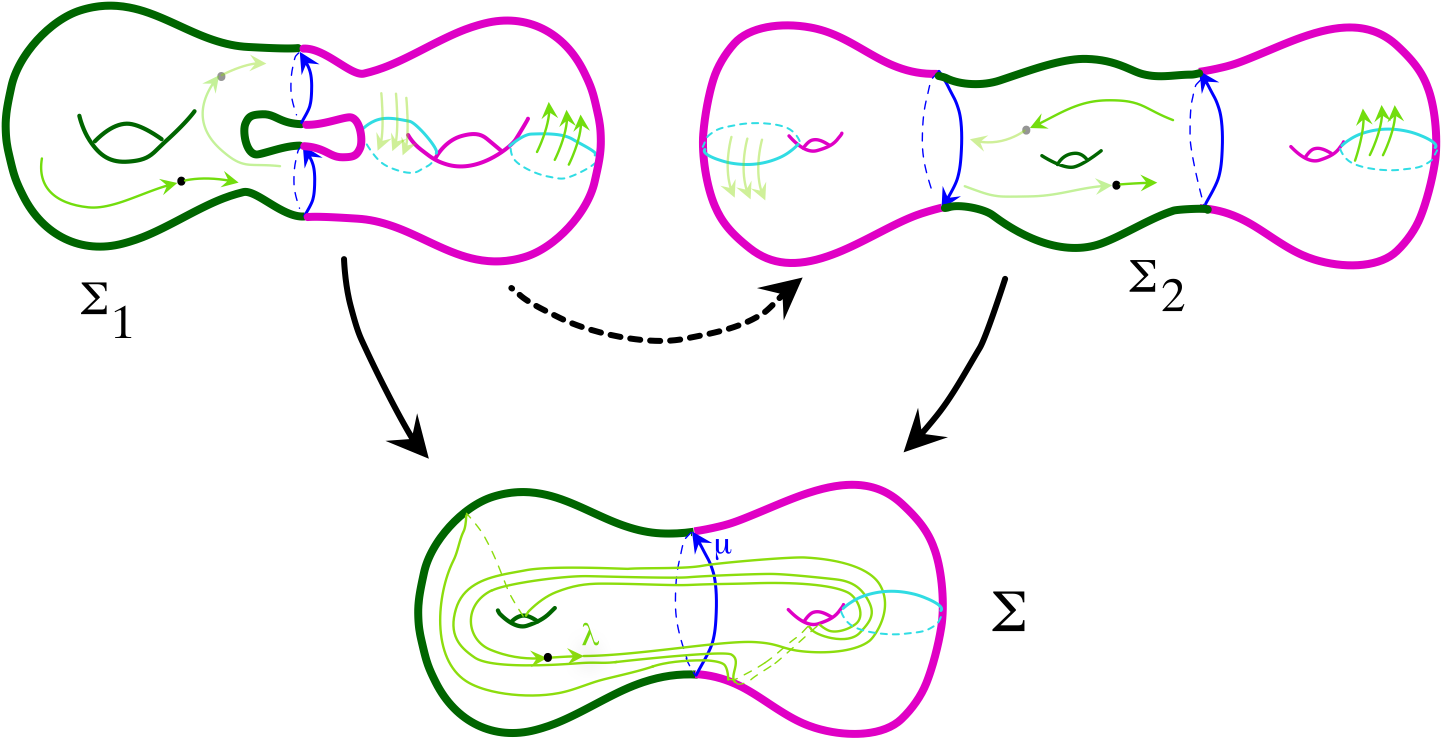}
             \caption{The two homeomorphic surfaces   $\Sigma_1$ (top left), $\Sigma_2$ (top right), and their (different) covering maps of  $\Sigma$ (bottom). The base point (black point) in $\Sigma$ has two lifts in each covers, in black and grey.  The curve $\mu$ in $\Sigma$ (in blue) has two lifts that are curves in each covers (in blue). The curve $\lambda$ in $\Sigma$ (in green) has two lifts that are {\it arcs} in each covers. For readability, in each cover,  only pieces of them are drawn, in green for one of the arcs (the one starting at the black point), and in light green for the other arc (the one starting at the grey point).   The  green lift  (that starts at the black point), ends either at the black or the grey point.  In the rightmost cover, it ends at the grey point, while in the leftmost cover, it ends at the black point, thus defining a curve. In particular, the whole preimage of green curve of $\Sigma$ is a single curve in $\Sigma_2$ on the right, but is the union of two curves in $\Sigma_1$ on the left.}
             \label{fig;covers}
         \end{figure}
         
         In both covers, the twist $\tau_{\mu}$ lifts as a mapping class that is trivial in homology.
           Observe that in $\Sigma_1$ the curve $\lambda$ lifts as two disjoint simple closed curves, but on $\Sigma_2$ it lifts  to two arcs, that when concatenated make a simple closed curve. In other words,  the squared twist $\tau_{\lambda}^2$ lifts to both covers, as a mapping class.
         
         In $\Sigma_2$, it lifts as a single Dehn twist over a simple (non-separating) curve. It follows that in this cover, the mapping class induced  by  $\phi $ induces, on the homology
         \emph{of $\Sigma_2$}, the identity plus a rank one endomorphism.
         


         In $\Sigma_1$, however,  $\lambda$ lifts as a pair of simple closed curves, and  $\tau_{\lambda}^2$ lifts  as a product of two Dehn twists over the two disjoint simple closed curves. It is not automatic that this product would be different in homology as a single Dehn twist, but it turns out that, by a slightly tedious but simple computation, the action on the homology of this mapping class is indeed the identity plus a rank 2 endomorphism. 
         
         Indeed, in homology of the cover $\Sigma_1$, one can form an explicit  basis, as \[(\alpha, \beta^2, \gamma^{-1}\beta^{-1}, \delta, \gamma^2, \beta\alpha^{-1}\beta^{-1}\alpha),\] and the matrix of the homology representation of $\tau_{\lambda}$ is the following matrix $M$ in this basis.   
           
         On the right, $M-I_6$ is of rank $2$.    
         \[ M=\begin{pmatrix} 
         -2 & -2 & -1  & -3 & 4 & 0 \\
         9 & 4  & 0 & 9 & -6 & 9 \\
         9 & 0  & -2 & 9 & 0 & 18 \\
         6 & 4  & 2 & 7 & -8 & 0 \\
         9 & 3  & 0 & 9 & -5 & 9\\
         0 & -1 & -1 & 0 & 2 & 4  \end{pmatrix}, \qquad M-I_6 =\begin{pmatrix} 
         -3 & -2 & -1  & -3 & 4 & 0 \\
         9 & 3  & 0 & 9 & -6 & 9 \\
         9 & 0  & -3 & 9 & 0 & 18 \\
         6 & 4  & 2 & 6 & -8 & 0 \\
         9 & 3  & 0 & 9 & -6 & 9\\
         0 & -1 & -1 & 0 & 2 & 3  \end{pmatrix} \]

         Therefore, the two lifts of $\phi$ to the two  surfaces of genus three are not conjugate by any isomorphism between the surfaces groups. 
         
         To conclude, our explicit example of a pair of pseudo-Anosov automorphisms on the fundamental group of a surface of genus 3, that are commensurably conjugate, but not conjugate,  are these two lifts of $\phi$. See Figure \ref{fig;covers}. These are the mapping classes  obtained, in each picture of genus 3 surface,  by twisting $k$ times over the pair of blue curves, and then twisting $2m$ times  along the pair of   green and orange curves, in $\Sigma_1$, and only $m$ times along the single  green-and-orange curve in $\Sigma_2$.
          \end{proof}

    \section{Infinitely many class number one pseudo-Anosovs}

    We start this Section by a lemma which is probably well known to specialists.  {Since we could not find an explicit reference, we include a proof for completeness.}
    
    \begin{lemma}
        There exists $V_0>0$, and  infinitely many conjugacy classes of maximal non-arithmetic uniform lattices in $\PSL_2(\bbC)$, with co-volume $\leq V_0$.
    \end{lemma}
    
    \begin{proof}
        Take $M_0$ a non-compact, finite volume hyperbolic 3-manifold with one cusp (a complement of a hyperbolic knot for instance), and consider a  sequence $M_n$  of  compact hyperbolic manifolds obtained by deeper and deeper hyperbolic Dehn filling on $M_0$. Their volumes $V_n$  are bounded  above by the volume $V_0$ of the initial manifold, and accumulate on $V_0$. It follows from a theorem of Borel  \cite[Theorem 8.2]{Bor81} that, after passing to  a subsequence, we can assume that they have non-arithmetic fundamental group $\Gamma_n$ in $\PSL_2(\bbC)$. By another theorem of Borel,  on commensurability \cite[Main Theorem]{Bor81}, for all these  $n$, there exists a unique biggest uniform lattice $\Gamma_n^{+}$ containing $\Gamma_n$, necessarily as a finite index subgroup. For each $n$, the co-volume of $\Gamma^{+}_n$ in $\bbH^3$  divides $V_n$, and is larger than a positive constant \cite{Gro_Bour, KM}. Therefore, since $V_n$ accumulates on $V_0$, the $\Gamma^{+}_n$ are eventually all of different co-volume, hence non-conjugate. 
        
    \end{proof}

    We finally prove Theorem \ref{theo;classnumberone}.
    
    \begin{proof}

    Let us denote by $\Lambda^0_i, i\in \bbN$ a sequence of  representatives of different  conjugacy classes of maximal non-arithmetic uniform lattices in $\PSL_2(\bbC)$.  
      In each $\Lambda^0_i$, by Agol's virtual fibration theorem {\cite[Theorem 9.2]{A}}, there is a finite index subgroup $\Lambda^1_i$ that maps onto $\bbZ$ with finitely generated kernel. Since the index of $\Lambda^1_i$ in $\Lambda^0_i$ is finite, the chains of strict inclusion of subgroups from $\Lambda^0_i$ to $\Lambda^1_i$ are finite. Therefore,  there exists a maximal subgroup $\Lambda_i$ of $\Lambda^0_i$, containing $\Lambda^1_i$, and surjecting on $\bbZ$ (under a homomorphism $\pi_i: \Lambda_i \to \bbZ$).  The index of  $\Lambda^1_i$ in $\Lambda_i$ is necessarily finite, less than $[\Lambda^1_i, \Lambda^0_i]$, and  the kernel of $\pi_i$ has therefore finite index in the fiber of $\Lambda^1_i$. In particular it is finitely generated. By Lemma \ref{lem;2orbi}, it is therefore
    the fundamental group of a hyperbolic $2$-orbifold. Let $G_i$ be this subgroup, and $t_i \in \Lambda_i$ such that  $\Lambda_i = G_i \rtimes \langle t_i\rangle$.

    Denote by $\phi_i$ the automorphism  of $G_i$ given by  conjugation by  $t_i$.
     \newcommand{\Mu}{{\rm M}}
     Let us compute the class number of $\phi_i$. Let $\psi_i$ be a commensurably conjugated automorphism. We want to prove that it is conjugate to $\phi_i$ in $\aut{G_i}$.  By Proposition \ref{prop;fiber_thickens}, the lattice $\Lambda_i$ lives in a larger lattice containing a copy $\Mu_i$ of the mapping torus of $\psi_i$ and surjecting onto $\bbZ$ with the fibers of $\Mu_i$ and $\Lambda_i$ in the kernel. By Borel's Theorem on commensurability \cite{Bor81},  since $\Lambda_i$ is non-arithmetic, there is a unique biggest lattice containing $\Lambda_i$, and it is therefore $\Lambda^0_i$. Therefore,  by maximality of $\Lambda_i$, $\langle \Mu_i\cup \Lambda_i\rangle = \Lambda_i$, and it follows from Lemma \ref{lem;rhoG_equal_psi_conjugate} that $\phi_i$ and $\psi_i$ are conjugated in $\aut{G_i}$. Therefore,  the class number of $\phi_i$ is one.
   
   The same argument also reveals that for $i\neq j$, $\phi_i$ and $\phi_j$ do not have any non-trivial power that are commensurably conjugate: if they did, the mapping tori would be represented in $PLS_2(\bbC)$ in a common lattice that fibers with finitely generated kernel, and Borel's Theorem again ensures that the maximal lattice containing them is unique, and therefore $i=j$. 
    \end{proof}
    
\section{Appendix: a discussion of the algebraic number theory for class number of matrices, and its analogous concepts}

     \subsection{Orders and ideal classes}
     
     We briefly discuss the algebraic number theory behind the class numbers of matrices in $\GL_n(\bbZ)$ with irreducible characteristic polynomial.
      Recall that an order is a ring that is finitely generated as $\bbZ$-module. 
      In an algebraic number field $K$ (a finite extension of $\bbQ$), there is a  maximal order which is the ring $\calO_K$ of algebraic integers. 
       Let $\chi$ be a unitary irreducible polynomial over $\bbZ$, and $\xi$ a root in $K=\bbQ[X]/(\chi)$. We are interested in the order $\bbZ[\xi]$ in $K$.  
     
     If $\fraka, \frakb$ are ideals in this ring, they are said to be equivalent if there exist $x,y$ such that $x\fraka = y\frakb$. The ideal $\fraka$ is invertible if there is $\frakb$ such that $\fraka \frakb$ is principal, hence equivalent to $(1)$. The monoid of equivalence classes of ideals is the ideal class  monoid (or semigroup), and restricting to the invertible ideals, this constitutes the Picard group of the order $\bbZ[\xi]$.    
     
     The Picard group of an order in a number field  is finite (see discussion in \cite{KlunersPauli}). The Ideal Class Monoid  of  $\bbZ[\xi]$  is finite \cite[Thm. 8]{ZZ}. It is a subset of  a finite union of Picard groups of orders containing  $\bbZ[\xi]$  (those orders are contained in the maximal order, the ring of algebraic integers). 
     When $n=2$, it is equal to this finite union of Picard groups (see \cite[Prop. 3.7]{Marseglia}). 
      The  Latimer-MacDuffee theorem says the following. 
     
     \begin{thm}\label{theo;LMD} \cite{LMD, T}
       Let $\chi$ be a unitary irreducible polynomial over $\bbZ$, and $\xi$ a root in $\overline \bbQ$.  Let $\calC_\chi$ be the set of $\GL_n(\bbZ)$-conjugacy classes of matrices in $\calM_n(\bbZ)$ whose characteristic polynomial is $\chi$. 
           Then $\calC_\chi$ is in bijection with the  Ideal Class Monoid  of  $\bbZ[\xi]$. 
     \end{thm}

     \subsection{Arbitrarily large class numbers}

     We owe to Sara Checcoli the following argument for showing that the class number is not bounded among elements of $\GL_2(\bbZ)$. 
     
     \begin{prop}
    The class numbers of elements in $\GL_2(\bbZ)$ with irreducible characteristic polynomials are finite but unbounded.   
     \end{prop}
     \begin{proof}
     We already know that they are finite by Theorem \ref{theo;LMD}. 
     A theorem of Montgomery and Weinberger \cite{MW} ensures that there are infinitely many square-free integers $d>0$ of the form $d=4n^2+1$ such that the class number of $\calO_{\bbQ(\sqrt{d})}$ is larger than $\sqrt{d} (\log\log d)/(\log d)$. 
     
     For each such $d$, consider Pell's equation $a^2-db^2=4$, which has integer solutions by Lagrange's theorem \cite{Lag}. 
       Let $a,b$ be a solution, and $\chi_{a}(X) = X^2 +aX +1$. Since $a^2 -4 = db^2$ is not a square, $\chi_a$ is irreducible over $\bbZ$, and it is obviously the characteristic polynomial of  $\left(\begin{smallmatrix} -a& -1\\ 1 & 0 \end{smallmatrix}\right)\in \GL_2(\bbZ)$. Its splitting field is $\bbQ(\sqrt{d})$, and since $d=4n^2+1\equiv   1\, mod \,4$, the ring of integers is $\bbZ[(1+\sqrt{d})/2]$. Its class number is larger than  $\sqrt{d} (\log\log d)/(\log d)$, and therefore, the class number of the order $\bbZ[(a + b\sqrt{d})/2]$ also. By the theorem of Latimer and MacDuffee,  for our integer $a$, the   class number  of $\left(\begin{smallmatrix} -a& -1\\ 1 & 0 \end{smallmatrix}\right)$ is larger than  $\sqrt{d} (\log\log d)/(\log d)$.
     \end{proof}

     \subsection{On the proof of Latimer-MacDuffee theorem}
     
     We reproduce a variation of the proof {given in} \cite{T}, that will help us connect to the case of surfaces.
       For $\xi \in \overline \bbQ$, the dual of the $\bbZ[\xi]$-module $\bbZ[\xi]^n$ is   
      $Hom_{\bbZ[\xi]}(\bbZ[\xi]^n, \bbZ[\xi])$, it is  the module of $\bbZ[\xi]$-linear combinations of the $n$ entries of vectors in $\bbZ[\xi]^n$.
       With the notation of Theorem \ref{theo;LMD}, it suffices to establish the following Proposition.  
     
     \begin{prop}\label{prop;LMD_hom}
       For $M \in \calM_n(\bbZ)$ of characteristic polynomial $\chi$, let  $v_M$ in $\bbZ[\xi]^n$ be an eigenvector for $\xi$, and  \[\fraka_{v_M} = \{ \alpha(v_M),  \alpha \in Hom_{\bbZ[\xi]}(\bbZ[\xi]^n, \bbZ[\xi])  \}.\]
       
       Then $\fraka_{v_M}$ is an ideal of $\bbZ[\xi]$,  the ideal class of $\fraka_{v_M}$ depends only on $M$, every class is realized, and for all $M,N$, the ideal classes  of $\fraka_{v_M}$ and  $\fraka_{v_N}$ are equal if and only if $[M]=[N]$ in $\calC_\chi$.
     \end{prop}
     
    \begin{proof}  
     {The module}
      $\fraka_{v_M}$ is trivially a $\bbZ[\xi]$-module. Therefore it is an ideal  of $\bbZ[\xi]$. Any two choices of $v_M$ differ by a scalar in the field of fractions $\bbQ[\xi]$. Observing  that  $\fraka_{\lambda v} = \lambda \fraka_{v}$, we obtain that $M$ thus determines an ideal class of $\bbZ[\xi]$. 
      
     Let us show that every ideal of $\bbZ[\xi]$ is realized in this way. If $\frakb$ is a non-zero ideal of $\bbZ[\xi]$, it contains and is contained in two rank $n$ free abelian groups (respectively $\bbZ[\xi]b $ for $b\in \frakb\setminus\{0\}$, and  $\bbZ[\xi]$), hence it is a free $\bbZ$-module of rank $n$.      
      Consider $(\omega_1, \dots, \omega_n)$ a $\bbZ$-basis.
     Being moreover a $\bbZ[\xi]$-module, each $\xi \omega_i$ is a unique integral combination of $\omega_1, \dots, \omega_n$. Let $M$ be the matrix whose rows are the integer coefficients of these combinations. Thus, $M$ is a square matrix for which $\xi$ is an eigenvalue, and therefore, by Irreducibility,  its characteristic polynomial is $\chi$.  Since $v_M=(\omega_1, \dots, \omega_n)^\top$ is an eigenvector, $\fraka_{v_M}$ contains $\frakb$, and is equal to it because the coordinate linear forms generate $Hom(\bbZ[\xi]^n, \bbZ[\xi])$. %
     
     Two conjugate matrices $M, PMP^{-1}$ (for $P\in \GL_n(\bbZ)$)  produce the same ideal class: indeed, the eigenvector $v_{PMP^{-1}}$ can be chosen as $Pv_M$, and since \[Hom_{\bbZ[\xi]}(\bbZ[\xi]^n, \bbZ[\xi]) = Hom_{\bbZ[\xi]}(\bbZ[\xi]^n, \bbZ[\xi]) P, \] the images  $\fraka_{v_M}$ and $\fraka_{Pv_M}$   are the same.
        Finally, if two matrices $M,N$ produce the same ideal class, let $v_M=(\omega_1, \dots, \omega_n)^\top$ and $v_N= (\eta_1,\dots, \eta_n)^\top$ be the respective eigenvectors in $\bbZ[\xi]^n$, scaled so that their entries generate the same ideal.  The tuples $(\omega_1, \dots, \omega_n)$ and $(\eta_1,\dots, \eta_n)$ are two bases of the same $\bbZ$-module, therefore there exists $P \in \GL_n(\bbZ)$ such that $Pv_M=v_N$.  Both matrices $N$ and  $PMP^{-1}$ have $v_N$ as eigenvectors for the eigenvalue $\xi$.  Since the entries of $v_N$ form a basis of the $\bbZ$-module they generate, there is a unique matrix $Q$ with integer entries such that $Q v_N = \xi v_N$. So, $Q = N= P MP^{-1}$, which finishes the proof. 
     \end{proof}
     
     \subsection{A connection to foliations: the torus} 
     
    Let $T=\bbZ^n \backslash \bbR^n$ 
    be the $n$-dimensional unit torus. 
     Let $\alpha \in Hom(\bbZ^n, \bbZ)$ be a non-trivial linear form, and $\ker_\bbZ \alpha$ its kernel. 
     Let $p$ be the smallest positive value of $\alpha (\bbZ^n)$.

     The group $\bbZ^n$ splits as $\ker \alpha \oplus \langle v_\alpha \rangle$ for $v_\alpha$ a preimage of $p$.   Consider the extension of $\alpha$ on $\bbR^n$,   the hyperplane  $\ker_{\bbR} \alpha$ in $\bbR^n$, and its image in $T$: this is an embedded  codimension 1  subtorus $T_\alpha$, the quotient of $\ker_\bbR \alpha$ by the translation group  $\ker_\bbZ \alpha$. 
     
     Let us say that $T_\alpha$ carries a weight of $p$ (the positive generator of $\alpha(\bbZ^n)$).
     Let $\calH$ be the space of weighted codimension 1 subtori that are images of kernels (in $\bbR$) of non-trivial forms in  $Hom(\bbZ^n, \bbZ)$.
    
    Let us define the intersection form $\iota: \calH \times \bbR^n  \to \bbN$ to be  \[ \iota(T_\alpha, v) = |\alpha(v)|. \]
     
    First this is well defined: the weighted codimension 1 subtorus $T_\alpha$ determines  the form $\alpha$ up to the sign.
      Let us justify the name. 
     If $v\in \bbZ^n$ which is not a proper multiple,  the image of the segment $[0, v] \subset \bbR^n$ in $T_\alpha$ is a simple closed curve. If $1 \in Im(\alpha)$, this curve intersects exactly $|\alpha(v)|$ times the codimension 1 subtorus $T_\alpha$, as seen when expressing $v$ in the sum  $\ker \alpha \oplus \langle v_\alpha \rangle$.  If now $p$ is the weight of $T_\alpha$, the quantity $|\alpha(v)|$ counts each intersection point $p$ times.

     For general $w\in \bbZ^n$,  $\iota(T_\alpha, v) = |\alpha(v)|$ counts this geometric weighted intersection number multiplied by the largest $k$ so that $w\in k \bbZ^n$. 
        If now $v\in \bbR^n$, for all $m$ let us choose a vector $v_m \in \bbR^m$ such that $\|v_m\|_\infty <1$,  $mv-v_m\in \bbZ^n$. Observe that  $\iota (T_\alpha, v) = \lim_m \frac{1}{m} \iota (T_\alpha, mv-v_m)$. The intersection form counts an average of intersection numbers for the multiples of $v$. This justifies the denominator.

     Geometrically, we associate, to an arbitrary vector $v\in \bbR^n \setminus \{0\}$, the image in the torus $T$ of the line $\bbR v$.  The foliation $\calF_v$ of the torus is then defined as the partition by the images of the parallel lines to $\bbR v$.  If $v\in \bbZ^n$, this is a foliation by closed curves,  and if $v \notin \bbZ^n$ each line of the foliation is a geodesic that is infinite in both directions.

     We may interpret the intersection form $\iota(\cdot, v)$ as defining  a measure $\mu$  on each torus $T_\alpha$ of weight $1$ (i.e. for $\alpha$ achieving $1$): for each Borel $\Omega$, $\mu(\Omega)$ is $\lim_m \frac 1m n(m) $ for $n(m)$ the number of intersection points of the image of $[0, mv]$ in $T_\alpha$, with $\Omega$ (counted with multiplicity if $v\in \bbZ^n$). The intersection form thus defines a transverse measure to the foliation $\calF_v$ that endows each  codimension 1 subtorus $T_\alpha$ (for $\alpha$ achieving $1$) of the total mass $\mu(T_\alpha)= \iota(T_\alpha, v)$.  Choosing a multiple of $v$ then multiplies the measure by the factor. The foliation has  a well defined projective transverse measure.

     Let now $M$ be a matrix in $\GL_n(\bbZ)$ as before,  and $v_M$ an eigenvector in  $\bbR^n$ for the eigenvalue $\xi$. 
        We  now relate this geometric intersection form to the considerations on the ideal class associated to the conjugacy class of $M$.
      
     \begin{lemma}\label{lem;homs}
     If, for $\gamma \in Hom_{\bbZ}(\bbZ^n, \bbZ) $,  one denotes by $\gamma'$  its linear extension to $\bbR^n$, one has  
     \[ \begin{array}{lcl} \{\alpha(v_M), \alpha \in Hom_{\bbZ[\xi]}(\bbZ[\xi]^n, \bbZ[\xi])  \} & =& \{\alpha(v_M), \alpha \in Hom_{\bbR}(\bbR^n, \bbR), \alpha(\bbZ^n)\subset \bbZ  \} \\ 
     &=&  \{ \gamma'(v_M), \gamma \in Hom_{\bbZ}(\bbZ^n, \bbZ) \} \end{array}\]
     \end{lemma}
     \begin{proof}
     The  equality between the two right hand sides is clear. In the first equality, the reverse inclusion is clear. To show the forward inclusion, let  $v_M=(v_1, \dots, v_n)$.  It suffices to show that for each $i$, $v_i$ and $\xi v_i$ are in $\{\alpha(v_M), \alpha \in Hom_{\bbR}(\bbR^n, \bbR), \alpha(\bbZ^n)\subset \bbZ  \}$. Of course $v_i$ is the scalar product $v_i = (e_i, v_M)$, and $\xi v_i$ is $v_i = (M^\top e_i, v_M)$. Both linear forms  $(e_i,  \cdot )$ and $(M^\top e_i, \cdot)$ send $\bbZ^n$ to $\bbZ$.
     \end{proof}

    With Lemma \ref{lem;homs} and Proposition \ref{prop;LMD_hom} one recognizes the following geometric interpretation, with notations as above.
    \begin{coro}
     The correspondence between $\calC_\chi$ and  the ideal class monoid of $\bbZ[\xi]$ is given by taking an eigenvector $v_M\in \bbZ[\xi]^n$ and  the  ideal of all intersection numbers of $v_M$ with weighted codimension 1 subtori of $T$ ---  or in other words, the ideal of all evaluations of  the transverse-measured foliation $\calF_{v_M}$ on  codimension 1 subtori. 
    \end{coro}

     \subsection{A view to laminations: higher genus surfaces}
     
     In the case of a closed orientable surface $\Sigma$, of genus $g\geq 2$, and  an automorphism of $\pi_1(\Sigma)$, one may define similar objects. For a pseudo-Anosov map $\phi$, the eigenvector $v_M$ (and its associated foliation) should be replaced by a geodesic  lamination $\lambda$ on the surface that is preserved by $\phi$, with a transverse measure $\tau$. ( Recall that $\tau$ is a measure on the arcs transverse to $\lambda$, such that $\tau$ is $0$ if the arc does not intersect $\lambda$.) Further, $\tau$  is multiplied by some characteristic value $\xi$ by the action of $\phi$. This transverse measure, when integrated on simple closed curves, provides an intersection number for the measured lamination.
     
     Consider a train track chart $\calT$ carrying $\lambda$. To describe the transverse measure $\tau$ it suffices to assign weights to each of the arcs of the train track, in such a way that at each node, a summation consistency holds.  The mapping class $\phi$ determines a map from 
     the underlying graph of $\calT$ to itself, and a matrix $M_{\phi,\calT}$ counting, in its columns,  the number of occurrences of each arc of $\calT$ in the image of each given arc of $\calT$. The vector $W$ of weights of each arc must satisfy $M_{\phi, \calT} W= \xi W$. Choosing $W$ in $\bbZ[\xi]^n$ thus allows us to define an ideal $\fraka_{M_{\phi, \calT}}$, as generated by the image of $W$ under the collection of intersection forms. 
     
    Since the intersection forms depend only on $\lambda$, and not $\calT$, we defined an ideal class of $\bbZ[\xi]$ associated to $\phi$. 
      If $\psi$ is conjugate to $\phi$ in the mapping class group of $\Sigma$, then its associated lamination is the image of that of $\phi$ by a conjugating mapping class, and therefore, they define the same ideal class.
      If $\phi$ and $\psi$ are commensurably conjugate, the stretch factor $\xi$ is the same for both, as easily seen in a finite cover of $\Sigma$.
    
   Completing a  picture analogous to the torus case would require understanding which ideal classes are realized, and understanding pseudo-Anosov  mapping classes defining the same ideal class, in a given mapping class group.

\end{document}